\documentclass{amsart}
\usepackage{graphicx,amssymb,amsmath,amsthm,hyperref,float,setspace,caption,amsrefs}

\newtheorem{theorem}{Theorem}[section]

\newtheorem{lemma}[theorem]{Lemma}
\newtheorem{proposition}[theorem]{Proposition}
\theoremstyle{definition}
\newtheorem{definition}[theorem]{Definition}
\newtheorem{remark}[theorem]{Remark}
\newtheorem{example}[theorem]{Example}
\newtheorem{question}[theorem]{Question}
\numberwithin{equation}{section}

\hypersetup{colorlinks=true,linkcolor=red,filecolor=blue,urlcolor=blue,citecolor=green}
\begin{document}

\title{When is arc crossing change an unknotting operation?}
\author{Zhiyun Cheng}
\address{School of Mathematical Sciences, Laboratory of Mathematics and Complex Systems, MOE, Beijing Normal University, Beijing, 100875, China}
\email{czy@bnu.edu.cn}

\author{Qian Liao}
\address{School of Mathematical Sciences, Beijing Normal University, Beijing 100875, China}
\email{liaoqian@mail.bnu.edu.cn}

\author{Jingze Song}
\address{School of Mathematical Sciences, Beijing Normal University, Beijing 100875, China}
\email{jzsong@mail.bnu.edu.cn}

\subjclass[2020]{57K10, 57M15}
\keywords{Unknotting operation, arc crossing change}

\begin{abstract}
This paper extends the study of arc crossing change, a local operation on knot diagrams recently introduced by Cericola in \cite{Cer2024}, from knot diagrams to link diagrams. We consider two types of arc crossing change on link diagrams and discuss when they are unknotting operations. Furthermore, we show that any two crossing points in an alternating knot diagram are arc crossing change admissible.
\end{abstract} 

\maketitle

\section{Introduction}\label{sec1}
In \cite{Shi2014}, Ayaka Shimizu proposed a new local operation on knot diagrams, called region crossing change. Roughly speaking, implement region crossing change on a region of a knot diagram switches all the crossing points on the boundary of this region. Ayaka Shimizu showed that for any knot diagram and any crossing point of it, there exist some regions such that if one takes region crossing changes on these regions, only this chosen crossing point is changed. It immediately follows that region crossing change is an unknotting operation for knot diagrams. Since then, the study of region crossing change has been extended to link diagrams \cites{CG2012,Che2013}, link diagrams on orientable surfaces \cites{DR2018,CCXZ2022} and link diagrams on non-orientable surfaces \cite{CS2025}.

If we regard the original crossing change as a 0-dimensional crossing change, then region crossing change can be considered as a 2-dimensional version crossing change. A logical follow-up question is how to define a 1-dimensional version crossing change. In \cite{Kin2022}, Kinuno introduced the notion of arc crossing change, which switches the two endpoints of a semi-arc if one implements arc crossing change on it. It turns out that this is also an unknotting operation for knot diagrams.

Very recently, Cericola considered another version of arc crossing change in \cite{Cer2024}, which applies to arcs of a knot diagram rather than semi-arcs. In other words, applying arc crossing change on an arc switches the two endpoints of this arc. From a geometric perspective, this operation makes the chosen arc longer. The arc crossing change introduced by Cericola is fundamentally different from all the unknotting operations mentioned above. Whether it is crossing change, arc crossing change introduced by Kinuno, or region crossing change, any two operations are commutative. In other words, if we are to carry out a series of operations, the order of these operations will not affect the final result. However, this is not the case for the arc crossing change introduced by Cericola. This phenomenon makes the study of Cericola's version of arc crossing change very fascinating, and its research methods are completely different from those of the aforementioned unknotting operations.

By using Gauss diagrams, Cericola established the following theorem.

\begin{theorem}[\cite{Cer2024}]\label{the1}
For a given a knot diagram, there exists a sequence of arc crossing changes that transforms this knot diagram into an ascending diagram.
\end{theorem}

As a corollary, arc crossing change of Cericola's version is also an unknotting operation for knot diagrams. From now on, whenever we refer to arc crossing change, we will always mean Cericola's version of arc crossing change. The following Figure \ref{j} shows how to use one arc crossing change to unknot the trefoil knot.

\begin{figure}[H]
\centering
\includegraphics[width=0.7\linewidth]{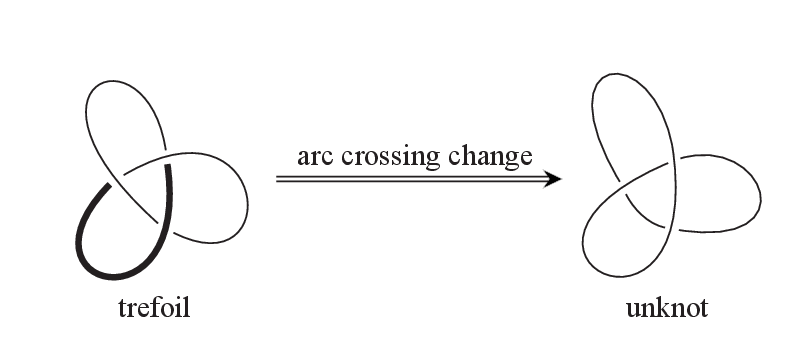}
\caption{Arc crossing change on a diagram of the trefoil knot}
\label{j}
\end{figure}

Naturally, this raises the question of whether this operation is also an unknotting operation for link diagrams. However, a new issue emerges: suppose that the two endpoints of an arc coincide at the same crossing, how does the crossing point change when an arc crossing change is applied on this arc? Within this framework, we can define two types of arc crossing changes for link diagrams, as illustrated in Figure \ref{a}.
\begin{figure}[H]
\centering
\includegraphics[width=0.80\linewidth]{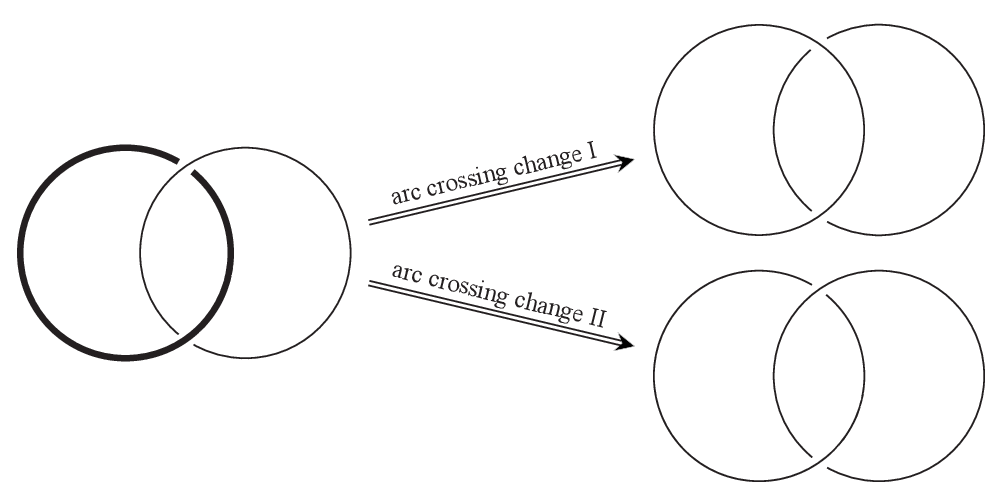}
\caption{Two types of arc crossing change on Hopf Link}
\label{a}
\end{figure}

We define arc crossing change I as follows: it changes the crossing that serves as both boundaries of the arc. Similarly, arc crossing change II is defined such that it does not change this crossing point, or equivalently, we can consider that this crossing point has been changed twice. 

In the present paper, link diagrams are always supposed to be connected, otherwise it suffices to consider each connected component of it. For a given link diagram $L$, let us use $C(L)$ to denote the set of all crossing points. A subset $S\subseteq C(L)$ is called \emph{arc crossing change admissible} if there exists a sequence of arc crossing changes which switches each crossing in $S$ but preserves all others. In the study of arc crossing changes, the underly 4-valent graph (shadow) of the link diagram is fixed. In other words, Reidemeister moves are forbidden during the transformation. Otherwise, it is easy to see that any crossing point $c$ of the diagram is arc crossing change admissible if one add a new crossing near $c$ by using the first Reidemeister move.

The following two theorems tell us when arc crossing change is an unknotting operation on link diagrams.

\begin{theorem}\label{the2}
Arc crossing change I is an unknotting operation for all link diagrams.
\end{theorem}

\begin{figure}[H]
\centering
\includegraphics[width=0.7\linewidth]{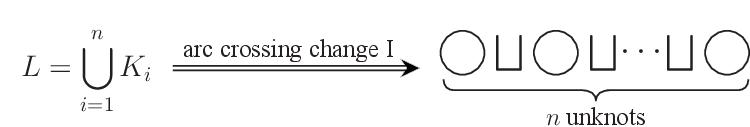}
\caption{Arc crossing change I on link diagrams}
\label{c}
\end{figure} 

\begin{theorem}\label{the3}
Let \(L=\bigcup_{i=1}^n K_i\) be a link diagram with \( n \) components, and denote the total linking number of \( L \) as \(t(L)=\sum_{1 \leq i < j \leq n} \mathrm{lk}(K_i, K_j)\), where \( \mathrm{lk}(K_i, K_j) \) denotes the linking number of components \( K_i \) and \( K_j \). The following statements hold for arc crossing change II:
\begin{enumerate}
\item If the total linking number \(t(L)\) is even, then arc crossing change II is an unknotting operation for \( L \).
\item If the total linking number \(t(L)\) is odd:
\begin{enumerate}
\item If one component of $L$ has a self-crossing, then arc crossing change II is an unknotting operation for \( L \).
\item If all components of $L$ are free of self-crossings, then arc crossing change II can transform \( L \) to a link diagram representing the disjoint union of \( n-2 \) unknots and a Hopf link.
\end{enumerate}
\end{enumerate}
\end{theorem}

\begin{figure}[H]
\centering
\includegraphics[width=0.83\linewidth]{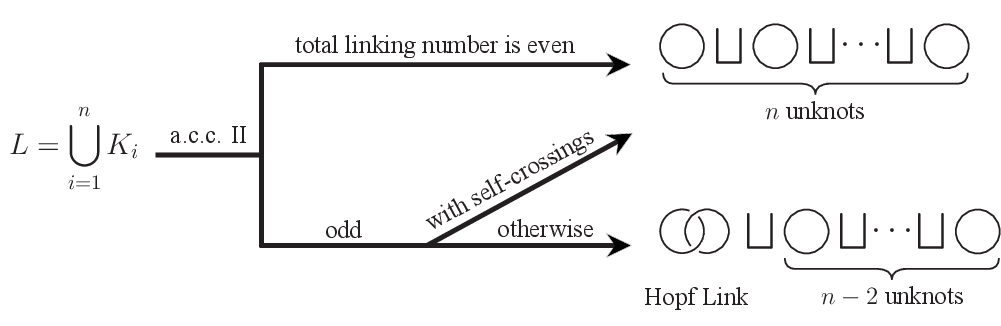}
\caption{Arc crossing change II on link diagrams}
\label{d}
\end{figure}

Compare the arc crossing change introduced by Kinuno and that introduced by Cericola, one observes that each arc crossing change of Cericola can be realized by that of Kinuno. In \cite{Cer2024}, Cericola asked whether the converse also holds. In other words, are these two different notions of arc crossing change actually equivalent? Later, we will show that this is not true in general. However, for some particular knot diagrams, these two types of arc crossing change are indeed equivalent. 

\begin{theorem}\label{the4}
Any two crossings of an alternating knot diagram with $n\geq2$ crossings are arc crossing change admissible.
\end{theorem}

The remainder of this paper is organized as follows. In Section \ref{sec2}, we present some concepts and corollaries pertaining to arc crossing change I and II. Section \ref{sec3} presents the proof of Theorem \ref{the2} and that of Theorem \ref{the3} respectively. Finally, we discuss the difference between the arc crossing change introduced by Cericola and that introduced by Kinuno. We provide a sufficient condition for two crossings being arc crossing change admissible and use it to give a proof of Theorem \ref{the4} in Section \ref{sec4}.

\section{Preliminaries}\label{sec2}
Prior to proving the main theorems, we first introduce some essential preliminary concepts for convenience. First, for a given link diagram with $n$ knot components (i.e., $L = \bigcup_{i=1}^n K_i$), we denote the set of all crossings formed by components $K_i$ and $K_j$ as $K_i \cap K_j$. Meanwhile, we denote the sub-link diagram of $L$ as $L_{\geq j}:=\bigcup_{i=j}^nK_i$. Similarly, the analogues of it, say $L_{>j}$ and $L_{\neq j}$ can be defined in a similar way. Next, we introduce some basic facts about link diagrams.

\begin{lemma}
Let $L$ be an $n$-component link diagram, there exists a labeling $K_1, \ldots, K_n$ of these components such that $K_i\cap L_{>i}\neq \varnothing$ for all $1 \leq i \leq n-1$.
\end{lemma}

We call such a labeling a \emph{felicitous labeling} of $L$. Note that felicitous labeling is not unique in general.

\begin{proof}
For a 2-component link diagram $L$, it is clear that labeling the two components as $K_1$ and $K_2$ satisfies the condition of felicitous labeling, since $L=K_1\cup K_2$ is connected.

Next, we proceed with induction on the component number $n$. Suppose the felicitous labeling exists for all $k$-component link diagram, now we consider a $(k+1)$-component link diagram $L$. According to the induction hypothesis, we can randomly select $k$ components that are connected and assign them a felicitous labeling $K_1', K_2', \ldots, K_k'$. We then label the remaining component as $K_1$ and relabel these $k$ components as $K_2, K_3, \ldots, K_{k+1}$. Since $K_1\cap\bigcup_{i=2}^{k+1}K_i\neq\emptyset$, now the labeling $K_1, K_2, \ldots, K_k, K_{k+1}$ is a felicitous labeling for $L$.
\end{proof}

\begin{definition}
Let $L = \bigcup_{i=1}^n K_i$ be a link diagram. For a given component $K_j$, we define two sets of crossings, called the \emph{over-crossing set} and \emph{under-crossing set} of $K_j$, as follows:
    $$
    \begin{aligned}
        &\mathcal{O}(K_j) := \{ c \in K_j\cap L_{\ne j}\mid K_j \text{ is the overstrand at } c \}, \\
        &\mathcal{U}(K_j) := \{ c \in K_j\cap L_{\ne j}\mid K_j \text{ is the understrand at } c \}.
    \end{aligned}
    $$
    It follows that the set of all crossings of $K_j$ excluding self-crossings can be expressed as the disjoint union of the over-crossing set and the under-crossing set, i.e., $K_j\cap L_{\ne j} = \mathcal{O}(K_j) \sqcup \mathcal{U}(K_j)$.
\end{definition}

\begin{figure}[H]
\centering
\includegraphics[width=0.6\linewidth]{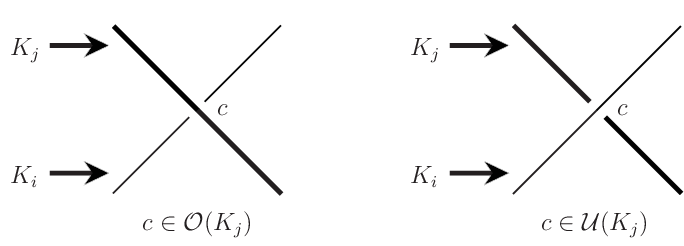}
\caption{A diagrammatic sketch of $\mathcal{O}(K_j)$ and $\mathcal{U}(K_j)$}
\label{g}
\end{figure}

Following the definition above, we introduce a notation that will be used throughout the remainder of this paper. For two distinct components \( K_i \) and \( K_j \) of a link diagram \( L \), we write \( K_i \precsim K_j \) if \( K_i \cap K_j \subseteq \mathcal{O}(K_j) \) (equivalently, \( K_i \cap K_j \subseteq \mathcal{U}(K_i) \)). Specifically, when the additional condition \( K_i \cap K_j \neq \varnothing \) is satisfied, we write \( K_i \prec K_j \).

\begin{definition}
For a link diagram \( L = \bigcup_{i=1}^n K_i \) with a felicitous labeling, a component \( K_j \) is called a \emph{simple component} if it has no self-crossings and satisfies \( L_{<j} \precsim K_j, K_j \precsim L_{>j} \).
\end{definition}

Obviously, a simple component represents an unknot unlinked with other components. Furthermore, we generalize the notion of ascending diagram from knot diagrams to link diagrams.

\begin{definition}
Let $L=\bigcup_{i=1}^n K_i$ denote a link diagram with a felicitous labeling. We say $L$ is an \emph{ascending diagram} if each component $K_i$ is an ascending diagram and there exists a permutation $\sigma\in S_n$ such that $$K_{\sigma(1)}\precsim K_{\sigma(2)}\precsim\cdots\precsim K_{\sigma(n)}.$$
\end{definition}

Figure \ref{f} below shows how to transform a link diagram of $3_1\cup4_1$ into an ascending diagram via arc crossing changes on the bold arcs.

\begin{figure}[H]
\centering
\includegraphics[width=0.7\linewidth]{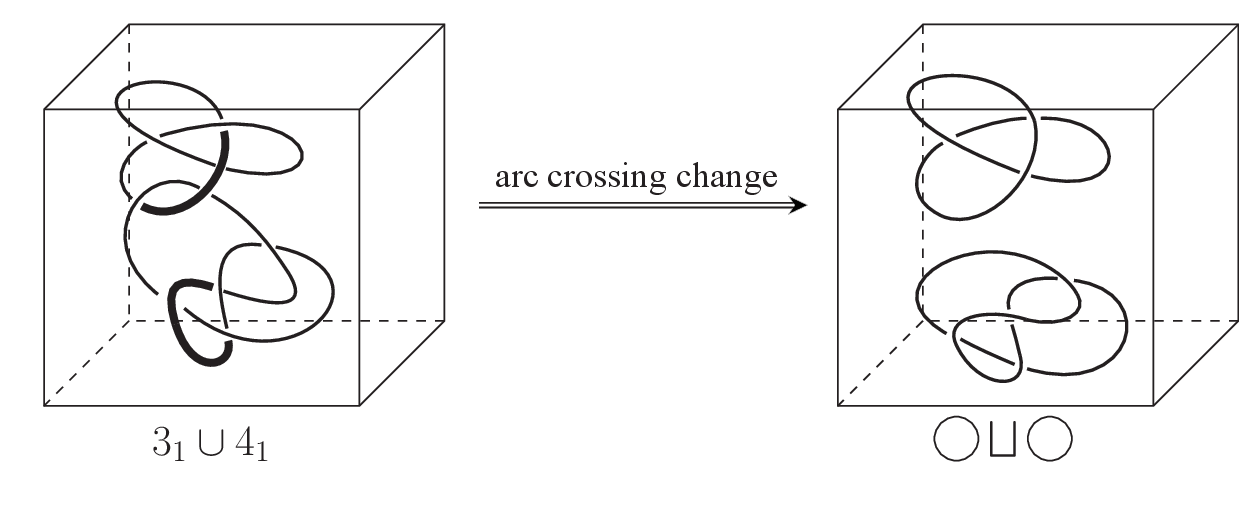}
\caption{Transform a link diagram of $3_1\cup4_1$ into an ascending diagram}
\label{f}
\end{figure}

\section{The proof of the main theorems}\label{sec3}
In this section, we give the proofs of Theorem \ref{the2} and Theorem \ref{the3}, respectively.
\subsection{The Proof of Theorem \ref{the2}}
\begin{proof}
First, we choose a labeling (not necessary felicitous) for the link diagram and denote it as $L=\bigcup_{i=1}^nK_i$. For a component \( K_j \) of the link diagram \( L \), let \( c \) be a crossing in \( K_i\cap K_j \) with \( i \neq j \). If \(c\in\mathcal{U}(K_j)\), applying arc crossing change I to one of the two understrands (it is possible that these two understrands coincide) switches the crossing $c$. Otherwise, if \(c\in\mathcal{U}(K_i)\), applying arc crossing change I to one of the two understrands converts $c$ into an overcrossing, as shown in Figure \ref{ee}. In either case, the crossing $c$ can be switched.

\begin{figure}[H]
\centering
\includegraphics[width=0.8\linewidth]{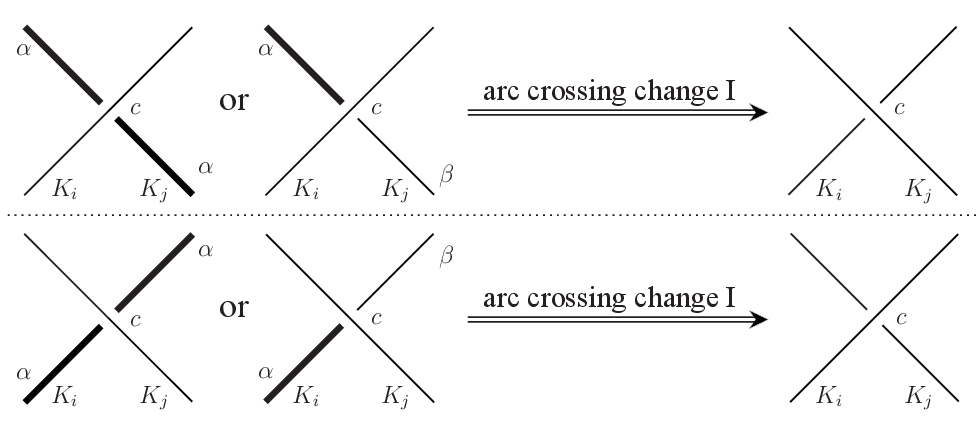}
\caption{Applying arc crossing change I to the bold arc}
\label{ee}
\end{figure}

Consider the component \( K_1 \), according to the discussion above, via a sequence of arc crossing changes I, all crossings of $\mathcal{U}(K_1)$ can be converted into overcrossings. Thus now we have $L_{\geq2}\precsim K_1$. By Theorem \ref{the1}, now \(K_1\) can be transformed into an ascending diagram via a sequence of arc crossing changes I. After this sequence of operations is completed, we still have $L_{\geq2}\prec K_1$ since only some self-crossings of $K_1$ are changed. Next, let us consider the component \(K_2\) and repeat the operations performed on \( K_1 \) for \( K_2 \). Then we have $L_{\geq3}\precsim K_1\precsim K_2$ and $K_2$ becomes an ascending diagram. Repeat this process for $K_3, \ldots, K_n$ one by one, finally we obtain a link diagram satisfying $K_1\precsim \ldots \precsim K_n$ and each $K_i$ $(1\leq i\leq n)$ is an ascending diagram. This is an ascending link diagram and the result follows.
\end{proof}

\subsection{The Proof of Theorem \ref{the3}}
Compared with arc crossing change I, arc crossing change II is not an unknotting operation in general. For instance, it is not an unknotting operation for the standard diagram of Hopf link, since this diagram is preserved under arc crossing change II. According to the definition of arc crossing change II, if each component of a link diagram has self-crossings, then there is no difference between these two types of arc crossing change. Arc crossing change II differs from arc crossing change I only when there exists a component $K_i$ that satisfies $K_i$ has no self-crossing and $|\mathcal{U}(K_i)|=1$. This thus inspires us to introduce the following concept.

\begin{definition}
Let \( L \) be an \( n \)-component link diagram. A component \( K_i \) that contains no self-crossings is referred to as a \emph{type C component} if $|\mathcal{U}(K_i)|=1$.
\end{definition}

\begin{figure}[H]
\centering
\includegraphics[width=0.7\linewidth]{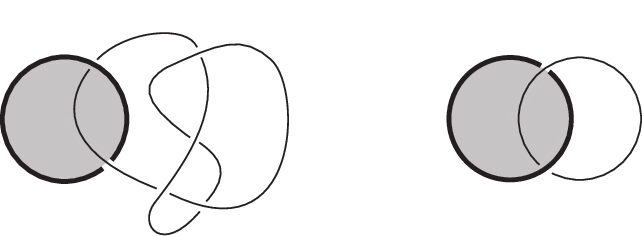}
\caption{Two examples of type C component}
\end{figure}

\begin{remark}
If there exists a knot diagram $K$ which has exactly one self-crossing and $\mathcal{U}(K)=\emptyset$, then applying arc crossing change I on $K$ switches this crossing but applying arc crossing change II preserves it. However, this is not essential since in either case $K$ represents the unknot.
\end{remark}

Furthermore, to facilitate the proof of Theorem \ref{the3}, we need the following lemmas.

\begin{lemma}\label{pro1}
Let $L$ denote an $n$-component link diagram. If a component $K_i$ of $L$ has no self-crossings, then the number of arcs of it equals the cardinality of $\mathcal{U}(K_i)$.
\end{lemma}

\begin{proof}
By definition, an arc is a part of $K_i$ between two undercrossing points. Since $K_i$ has no self-crossing point, it follows that the number of arcs of $K_i$ is equal to the cardinality of $\mathcal{U}(K_i)$.
\end{proof}

\begin{lemma}\label{pro2}
Let $L = \bigcup_{i=1}^n K_i$ be a link diagram. For any component $K_i$ $(1\leq i\leq n)$, the cardinality of $\mathcal{U}(K_i)$ and the sum of linking numbers of $K_i$ with other components have the same parity, i.e.
\[
|\mathcal{U}(K_i)|=\sum_{j\neq i}\text{lk}(K_i, K_j) \pmod{2}.
\]
\end{lemma}

\begin{proof}
It follows directly from the definition of linking number and that of $\mathcal{U}(K_i)$.
\end{proof}

\begin{lemma}\label{lem2}
For a link diagram $L$, if each component has no self-crossings, then the parity of the total linking number of $L$ is invariant under arc crossing changes II.
\end{lemma}

\begin{proof}
Since each component of $L$ has no self-crossings, applying arc crossing change II to an arc of $K_i$ alters exactly two crossings, say $c_1 \in K_i\cap K_j$ and $c_2 \in K_i\cap K_k$ where $j, k\neq i$. If $c_1\neq c_2$, this operation affects only the parities of linking numbers $\text{lk}(K_i, K_j)$ and $\text{lk}(K_i, K_k)$. Note that it is possible $K_j=K_k$. If $c_1=c_2$, i.e. $K_i$ is a type C component, then the link diagram $L$ is preserved. In either case the parity of the total linking number is preserved.
\end{proof}

Now we turn to the proof of Theorem \ref{the3}.

\begin{proof}
Let $L=\bigcup_{i=1}^nK_i$ denote a link diagram with a felicitous labeling. 

\begin{enumerate}
\item First, we consider the case that each component of $L$ has no self-crossings. 
\begin{enumerate}
\item $n=2$ : If $K_1\prec K_2$, then both $K_1, K_2$ are simple components and $L$ represents an unlink. Consider the case $\mathcal{O}(K_1)=\mathcal{U}(K_2)\neq\emptyset$, if lk$(K_1, K_2)$ is even, then $K_2$ is not a type C component and applying arc crossing change II to any arc of $K_2$ reduces $|\mathcal{O}(K_1)|$ by two. Repeat this operation several times one turns $|\mathcal{O}(K_1)|$ into zero and turns $L$ into the unlink. If lk$(K_1, K_2)$ is odd, according to Lemma \ref{pro2}, we have $|\mathcal{U}(K_2)|=|\mathcal{O}(K_1)|$ is odd. Since applying arc crossing change II to an arc of $K_2$ reduces $|\mathcal{O}(K_1)|$ by two, finally we obtain a new link diagram with $|\mathcal{O}(K_1)|=|\mathcal{U}(K_2)|=1$ and thus $K_2$ becomes a type C component. It follows that now $L$ represents the Hopf link.
\item $n\geq3$ : Consider the first component \( K_1 \), if \( K_1 \prec L_{>1} \), then \( K_1 \) is already a simple component. If not, this implies that some crossings in \( K_1 \cap L_{>1} \) belong to \(\mathcal{O}(K_1)\). Then, we claim that we can find a sequence of arc crossing changes II applied to \( L_{>1} \) to convert all the aforementioned overcrossings into undercrossings. For a component $K_i$ $(i\geq2)$ with \((K_1\cap K_i)\cap\mathcal{O}(K_1)\neq\emptyset\), if \( K_i \) is not a type C component, then applying arc crossing change II to the arc of \( K_i \) that is incident to the target over-crossing of \( K_1 \) switches the crossing. Otherwise, if \( K_i \) is a type C component, since \( n\geq 3 \), there exists another component \( K_j \) such that \( K_i \cap K_j \neq\emptyset\). Notice that $K_i$ has only one undercrossing since it is a type C component, it follows that $K_i\cap K_j\subseteq\mathcal{O}(K_i)$. Now we select a crossing \( c\in K_i \cap K_j \), and apply an arc crossing change II to one understrand of \( c\). Since \( K_i \) has exactly one arc, this operation will switch \( c \), and after the operation \(K_i\) is no longer a type C component. Notice that this operation does not increase $|\mathcal{O}(K_1)|$. Thus, repeating the aforementioned operations yields the result \( K_1 \prec L_{>1} \), i.e., the component \( K_1 \) becomes a simple component. Similarly, the remaining components \( K_2, K_3, \ldots, K_{n-2} \) can be transformed into simple components sequentially. During this process, since $L_{<j} \precsim L_{>j}$, arc crossing changes II applied to arcs of \(L_{>j}\) that reduce $|(K_j\cap L_{>j})\cap\mathcal{O}(K_j)|$ do not alter any crossing of \( L_{<j} \). Since \( K_1, \cdots, K_{n-2} \) are all simple components, it follows that $\text{lk}(K_i, K_j)=0$ for any $1\leq i\leq n-2$ and $j\neq i$. We continue our discussion by looking at two scenarios.
\begin{enumerate}
\item If the total linking number is even, since only the linking number between $K_{n-1}$ and $K_n$ has nontrivial contribution to the total linking number, together with Lemma \ref{lem2}, it follows that $\text{lk}(K_{n-1}, K_n)=0 \pmod{2}$. According to Lemma \ref{pro1} and Lemma \ref{pro2}, this implies that the number of arcs of \( K_n \) is even. Just like the 2-component case above, finally all components of $L$ can be converted into simple components and hence the link diagram represents an unlink. 
\item If the total linking number is odd, then \(\text{lk}(K_{n-1}, K_n)=1 \pmod{2} \) and by Lemma \ref{pro1} and Lemma \ref{pro2}, the number of arcs of \( K_n \) is odd. Thus, applying arc crossing changes II to \( K_n \) can transform it into a type C component  where \( L_{<n-1} \precsim K_{n-1}\cup K_n \), and \(K_{n-1}\cup K_n \) represents the Hopf link.
\end{enumerate}
\end{enumerate}
The proof for the case (1) is finished.
\item Next, we consider the case that at least one component of \( L \) has self-crossings. 
\begin{enumerate}
\item If every component has self-crossings, then there are no type C components. As we mentioned before, in this case there is no essential difference between the two types of arc crossing change. The result follows directly from Theorem \ref{the2}.
\item For the weaker scenario, if some (but not all) components have self-crossings whereas the rest have no self-crossings, we can decompose \( L \) into the disjoint union of two of sub-links, denoted as $L=\mathcal{L}_c\cup\mathcal{L}_s$. Here, $\mathcal{L}_c$ consists of all the knot components that have no self-crossings and $\mathcal{L}_s$ consists of all other components. For each knot component in $\mathcal{L}_c$, let us consider the maximal connected component (as a link diagram on the plane) containing this component, then we can write $\mathcal{L}_c$ as the disjoint union of maximal connected components, say $\mathcal{L}_c=L_{c, 1}\cup L_{c, 2}\cup\cdots\cup L_{c, p}$. In other words, for each $1\leq i\leq p$, every knot component of $L_{c, i}$ has no self-crossings, as a link diagram on the plane $L_{c, i}$ is connected, and $L_{c, i}\cap L_{c, j}=\emptyset$ if $i\neq j$. Similarly, we write $\mathcal{L}_s=L_{s, 1}\cup L_{s, 2}\cup\cdots\cup L_{s, q}$.

Now what we want to do is to make $L_{s, i}\precsim L_{c, j}$ for all $L_{s, i}\subseteq\mathcal{L}_s$ and all $L_{c, j}\subseteq\mathcal{L}_c$ via a sequence of arc crossing changes II, except some very special cases. Choose a knot diagram $K\subseteq L_{s, i}$ and a knot diagram $K'\subseteq L_{c, j}$ such that $(K\cap K')\cap\mathcal{O}(K)\neq\emptyset$, now there are three possibilities:
\begin{enumerate}
\item $K'$ is not a type C component. In this case, for any $c\in(K\cap K')\cap\mathcal{O}(K)$ we can apply arc crossing change II to one of the understrand of $c$ to switch $c$. Then $|(K\cap K')\cap\mathcal{O}(K)|$ is reduced by 1, or 2 if the other endpoint of this understrand also belongs to $(K\cap K')\cap\mathcal{O}(K)$. Repeat this process until all the crossings of $(K\cap K')\cap\mathcal{O}(K)$ have been switched and finally we have $K\precsim K'$.
\item $K'$ is a type C component but not a maximal connected component. Since $K'$ is not a maximal connected component, we can find another knot diagram $K''\subseteq L_{c, j}$ such that $K'\cap K''\neq\emptyset$. On the other hand, since $K'$ is a type C component and $(K\cap K')\cap\mathcal{O}(K)\neq\emptyset$, it follows that $K'\cap K''\subseteq\mathcal{U}(K'')$. Choose a crossing $c\in K'\cap K''$ and apply arc crossing change II to one of the understrand of $c$, then $c$ will be switched and $K'$ is not a type C component anymore. Then we can use the method above to handle $K'$. Notice that after step (i), the other endpoint of the chosen understrand is not the undercrossing of any component of $\mathcal{L}_s$, therefore it does not affect the components that have been handled in step (i).
\item $K'$ is a type C component and also a maximal connected component. In other words, $K'=L_{c, j}$ which represents an unknot. In this case, we do nothing.
\end{enumerate}

Now if $L_{c, j}$ is not a single component of type C as case (iii) mentioned above, we have $\mathcal{L}_s\precsim L_{c, j}$. Next we want to convert each such $L_{c, j}$ into an unlink via arc crossing changes II. Note that a maximal connected component which consists of just one component of type C represents the unknot, therefore we do not need to do anything on it. For other cases, since each component of $L_{c, j}$ has no self-crossings, there are two possibilities as we have met in case (1) above.
\begin{enumerate}
\item If the total linking number of $L_{c, j}$ is even, similar to case (1) we can convert $L_{c, j}$ into an unlink via arc crossing changes II.
\item If the total linking number of $L_{c, j}$ is odd, we claim that we can also convert $L_{c, j}$ into an unlink via the help of some component of $\mathcal{L}_s$. First, we can use the method in case (1) to convert $L_{c, j}$ into a link diagram representing the disjoint union of some unknots and a Hopf link via a sequence of arc crossing changes II. Without loss of generality, we choose a felicitous labeling of $L_{c, j}$ as $L_{c, j}=K_1\cup\cdots\cup K_u$ and now we have $K_1\precsim K_2\precsim\cdots\precsim K_{u-1}\cup K_u$, where $K_i$ $(1\leq i\leq u-2)$ is a simple component, $|\mathcal{O}(K_1)|=|\mathcal{U}(K_2)|=1$ and therefore $K_{u-1}\cup K_u$ represents the Hopf link.

Denote the single-point set $\mathcal{U}(K_u)=\{c\}$, obviously switching $c$ turns $K_{u-1}\cup K_u$ into a 2-component trivial link. Since the link diagram $L$ is connected and $L_{c, i}\cap L_{c, j}=\emptyset$ if $i\neq j$, it follows that there exists a component $J\in \mathcal{L}_s$ such that $J\cap L_{c, j}\neq\emptyset$. Let us assume that $J\cap K_{u-v}\neq\emptyset$ for some $v<u$ and choose a crossing $c''\in J\cap K_{u-v}$. Choose an understrand of $c''$, say $a_1$, and denote the other endpoint of $a_1$ by $c'$. Now let us walk along one side of the overstrand of $c''$, say $a_2$, until we meet an undercrossing. The existence of this undercrossing is guaranteed by the fact that the labeling $L_{c, j}=K_1\cup\cdots\cup K_u$ is felicitous, which means that $K_{u-v}\cap(K_{u-v+1}\cup\cdots\cup K_u)\neq\emptyset$. We might as well assume that the component we encounter is exactly $K_{u-v+1}$. Then we choose one side of the overstrand of this crossing and walk along $a_3$ until we meet the next undercrossing. Repeat this process until we meet the component $K_u$. Since $|\mathcal{U}(K_u)|=1$, $K_u$ has just one arc. Choose one side of the overstrand of the crossing where we meet $K_u$ and walk along $a_{v+2}$ until we arrive at the crossing $c$. See Figure \ref{HF} for a simple example.

\begin{figure}[H]
\centering
\includegraphics[width=0.8\linewidth]{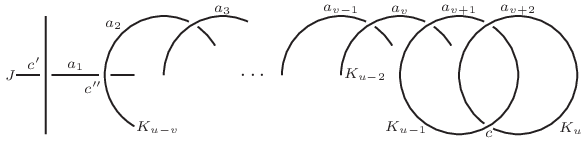}
\caption{Unlink the Hopf link}
\label{HF}
\end{figure}

Now let us apply arc crossing changes II to the arcs $a_1, a_2, \ldots, a_{v+2}$ serially. Note that $a_1$ is an arc of $L$, but $a_2$ is not. However, after applying arc crossing change II to $a_1$, $a_2$ becomes an arc in the new link diagram. Finally, only the two crossing points $c$ and $c'$ are changed and now the link diagram $L_{c, j}$ represents an unlink. Note that it is possible $J\precsim L_{c, j}$ does not hold anymore, since the overstrand of $c'$ might be of some component of $L_{c, j}$. After switching $c'$, it becomes an undercrossing of this component.
\end{enumerate}

Repeat the process above for all $L_{c, j}$ $(1\leq j\leq p)$ except those consist of only one knot component of type C, now every maximal connected component of $\mathcal{L}_c$ represents an unlink. Finally, recall that each knot component of $\mathcal{L}_s$ has self-crossings, therefore none of it is a type C component. Now for any crossing point $c\in K\cap K'$, where $K\subseteq\mathcal{L}_c, K'\subseteq\mathcal{L}_s$, if $c\in\mathcal{O}(K)$, then it can be switched by applying arc crossing change II to an understrand of $c$. In the end, we have $L_{c, j}\precsim L_{s, i}$ for all $L_{c, j}\subseteq\mathcal{L}_c$ and all $L_{s, i}\subseteq\mathcal{L}_s$. Note that after these operations, each $L_{c, j}$ still represents an unlink. Since each component of $L_{s, i}$ has self-crossings, as we have discussed before, $L_{s, i}$ can be converted to the unlink via a sequence of arc crossing changes II. In the end, we obtain a link diagram which represents the unlink from $L$ via finitely many arc crossing changes II.
\end{enumerate}
The proof for the case (2) is finished.
\end{enumerate}
\end{proof}

\section{Arc crossing change on alternating knot diagrams}\label{sec4}
\subsection{Directed graph derived from a knot shadow}
Given a knot diagram $K$, recall that a \emph{semi-arc} is a piece of $K$ between two adjacent crossing points, and an \emph{arc} is a piece of $K$ between two undercrossing points with no undercrossing points in between. Arc crossing change of Kinuno's version is taken on semi-arcs, while Cericola's version is applied on arcs. Obviously, Cericola's arc crossing change on an arc can be realized by that of Kinuno on the semi-arcs of this arc. In \cite{Cer2024}, Cericola asked whether the converse still holds, i.e. whether each Kinuno's arc crossing change can be realized by a sequence of arc crossing changes introduced by Cericola? Next, we will explain that this matter is generally not valid.

For a knot diagram $K$ with $n$ crossings, if we ignore the over/undercrossing information then we obtain a 4-valent plane graph $G$, called the \emph{shadow} of $K$. Obviously, by assigning an over/undercrossing  to each vertex of $G$, we can derive $2^n$ different knot diagrams from the graph $G$. We can establish a one-to-one correspondence between $\{0, 1\}^n$ and these knot diagrams as follows. Consider the checkerboard coloring of $K$, which is an assignment of the color black or white to each region of $K$ such that adjacent regions receive different colors and the unbounded region is assigned with white. We assign $0$ to a crossing point of $K$ if the region swept when the overstrand rotates counterclockwise to the understrand is black. Otherwise, we assign $1$ to it. Then for a fixed knot shadow, each element of $\{0, 1\}^n$ corresponds to a knot diagram. In particular, $(0, \ldots, 0)$ and $(1, \ldots, 1)$ are alternating and mirrors of each other.

For a given knot shadow $G$, in order to have a better understanding of the effect of arc crossing changes on knot diagrams, it is convenient to introduce an associated directed graph $A_G$. A similar construction of undirected graph for region crossing change can be found in \cite{CCXZ2022}. The graph $A_G$ has $2^n$ vertices, each vertex corresponds to an element of $\{0, 1\}^n$, hence also a knot diagram of $G$. There is an edge directed from vertex $v$ to vertex $w$ if and only if the knot diagram corresponding to $w$ can be obtained from the knot diagram corresponding to $v$ by one arc crossing change (of Cericola's version).

\begin{example}\label{4.1}
Let $G$ denote the shadow of the standard diagram of the figure-8 knot. The associated directed graph $A_G$, which has two connected components, is depicted in Figure \ref{figure-8}. Here both $0000$ and $1111$ represent the figure-8 knot, since it is amphichiral. Note that the knot diagram $0000$ can be obtained from $1100$ by an arc crossing change introduced by Kinuno. However, since $0000$ has indegree $0$, the knot diagram $0000$ cannot be obtained from $1100$ via Cericola's arc crossing changes. Therefore, the arc crossing change introduced by Kinuno and that introduced by Cericola are not equivalent.

\begin{figure}[h]
\centering
\includegraphics[width=0.75\linewidth]{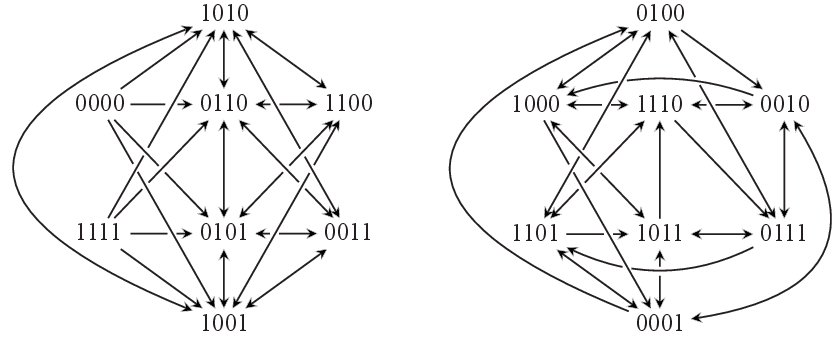}
\caption{The directed graph of the shadow of figure-8 knot}
\label{figure-8}
\end{figure}
\end{example}

Here we list some elementary facts about $A_G$. The proofs are left to the reader.

\begin{proposition}
Let $G$ be the shadow of a knot diagram with $n$ crossing points and $A_G$ the associated directed graph, then the following statements hold true.
\begin{enumerate}
\item Each vertex of $A_G$ has outdegree $n$;
\item For any vertex $v$ of $A_G$, denote the arcs of the corresponding knot diagram by $a_1, \ldots, a_n$. For an arc $a_i$ $(1\leq i\leq n)$, if there are $m_i$ crossing points having $a_i$ as their overstrands, then the indegree of $v$ equals to $\sum_{i=1}^n\binom{m_i}{2}$.
\item $A_G$ has at least two connected components.
\end{enumerate}
\end{proposition}

As a corollary, if $m_i\leq 1$ for all $1\leq i\leq n$, then the vertex $v$ has indegree 0. The figure-8 knot diagram in Example \ref{4.1} falls exactly into this situation. It follows that this kind of knot diagram (e.g. alternating knot diagram) cannot be obtained from any other knot diagram with the same shadow by arc crossing changes. 

Given a knot diagram, now we know that in general not every two crossing points are arc crossing change admissible. The rest of this section is devoted to show that for some special knot diagrams, for example alternating knot diagrams, this statement holds true.

\subsection{A proof of Theorem \ref{the4}}\label{4.3}
\begin{definition}
Let $K$ be a knot diagram and $G$ the underlying 4-valent plane graph, for two crossing points $x, y\in C(K)$, we say an $x-y$ trail $T$ in $G$ is \emph{admissible}, if $T$ begins with an understrand of $x$ and ends with an understrand of $y$, and when $T$ meets a crossing $v\in C(K)\setminus\{x, y\}$, then $T$ either goes straight along the overstrand or turn left/right at an overcrossing/undercrossing. See the dotted lines depicted in Figure \ref{walk}.

\begin{figure}[h]
\centering
\includegraphics{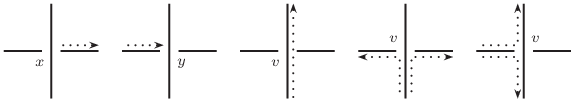}
\caption{When an $x-y$ trail meets a crossing}
\label{walk}
\end{figure}
\end{definition}

In order to prove Theorem \ref{the4}, we need the following lemma.

\begin{lemma}\label{4.4}
If two crossing points $x, y$ of a knot diagram $K$ can be connected by an admissible $x-y$ trail, then $x, y$ are arc crossing change admissible.
\end{lemma}
\begin{proof}
By ignoring the crossings in the trail that are passed through in a straight line, let us consider the vertices in this trail where the trail turns left or right when passing through them, and denote them by $v_1, \ldots, v_k$. Walking along $T$ beginning with $x$, if each $v_i$ ($1\leq i\leq k$) is an undercrossing when we meet it, then by applying arc crossing changes to the arcs between $x$ and $v_1$, $v_1$ and $v_2$, $\cdots$, $v_k$ and $y$ one by another, we obtain a new knot diagram where only $x$ and $y$ are switched. Note that in $K$, the piece of $K$ between $v_1$ and $v_2$ is not an arc, however after applying arc crossing change to the arc between $x$ and $v_1$, it becomes an arc in the newly obtained knot diagram.

If not all $v_1, \ldots, v_k$ are undercrossings, without loss of generality, we assume $v_i$ is the first overcrossing and $v_{i+1}, \ldots, v_{i+j}$ are all overcrossings, but $v_{i+j+1}$ is an undercrossing. After applying arc crossing changes to the arcs between $v_{i+j}$ and $v_{i+j+1}$, $v_{i+j-1}$ and $v_{i+j}$, $\cdots$, $v_i$ and $v_{i+1}$, we obtain a new knot diagram where $v_i$ becomes an undercrossing and $v_{i+j+1}$ becomes an overcrossing. See Figure \ref{OU} for an example of $j=2$. 

\begin{figure}[h]
\centering
\includegraphics{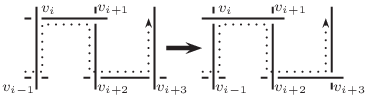}
\caption{Switching $v_i$ and $v_{i+j+1}$}
\label{OU}
\end{figure}

Repeat this progress until $v_1, \ldots, v_l$ are undercrossings and $v_{l+1}, \ldots, v_k$ are overcrossings. Now there exists an admissible trail from $y$ to $v_{l+1}$ and all the crossings in between are undercrossings. According to the discussion above, $y$ and $v_{l+1}$ are arc crossing change admissible. After switching $y$ and $v_{l+1}$, we find that now $x$ and $v_{l+1}$ are arc crossing change admissible since there exists an admissible trail from $x$ to $v_{l+1}$ and all crossings in between are undercrossings. Finally, we obtain a new knot diagram of which only $x$ and $y$ have been changed, as desired.
\end{proof}

Now we turn to the proof of Theorem \ref{the4}.

\begin{proof}
Let $K$ be an alternating knot diagram and $G$ the underlying 4-valent plane graph, we can make $G$ into a two-in two-out directed graph as in Figure \ref{2i2o}. The assumption that $K$ is alternating ensures the rationality of the orientation of each edge in graph $G$.

\begin{figure}[h]
\centering
\includegraphics{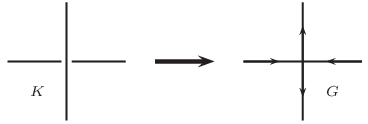}
\caption{Obtain a directed graph $G$ from an alternating diagram $K$}
\label{2i2o}
\end{figure}

Choose two vertices of $G$, denoted by $x$ and $y$. Let us use $x'$ and $x''$ to denote the two vertices incident to $x$ which are directed toward $x$. Note that $x'\neq x''$ since $K$ is a knot diagram. 

We claim that there exists a directed trail in $G$ with first vertex $x'$ or $x''$ and last vertex $y$ which does not pass through the vertex $x$. Let us proceed by induction on the crossing number $n$.

When $n\leq4$, here we list some examples in Figure \ref{234}. For the first and second directed graphs, we have $x'=y$. For the third directed graph with four vertices, the edge $a_2$ connecting $x'$ and $y$ provides the trail desired. For other cases of directed graph with at most four vertices or other choices of $\{x, y\}$, it can be verified similarly.

\begin{figure}[h]
\centering
\includegraphics{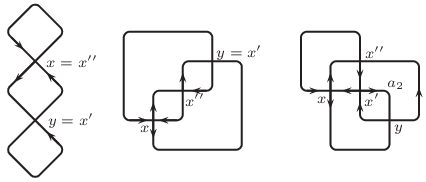}
\caption{Directed graphs with 2, 3, 4 vertices}
\label{234}
\end{figure}

Now we assume that $n\geq5$ and the claim holds true for all alternating knot diagrams with at most $n-1$ crossings and any chosen pair of vertices. Consider an alternating knot diagram $K$ with $n$ crossing points, since $n\geq5$, we can choose a vertex $z$ of $G$ such that $z\notin\{x, x', x'', y\}$. Now we resolve the vertex $z$ as in Figure \ref{z}. Notice that there are two types of resolutions and we choose the one such that new diagram with crossing number $n-1$ is also a knot diagram.

\begin{figure}[h]
\centering
\includegraphics{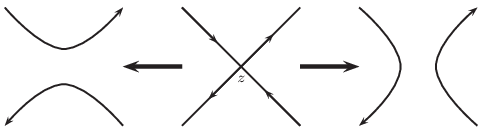}
\caption{Resolution of the vertex $z$}
\label{z}
\end{figure}

Now the new 4-valent plane graph is a two-in two-out directed graph with $n-1$ vertices. By assumption, there exists a directed trail with first vertex $x'$ or $x''$ and last vertex $y$ which does not pass through $x$. Whether the vertex $z$ is involved in this trail or not, this directed trail corresponds to a directed trail in $G$ connecting $x'$ or $x''$ and $y$ which does not pass through $x$. We complete the proof of the claim.

Without loss of generality, we assume that there exists a trail with first vertex $x'$ and last vertex $y$. Let us use $a_2, \ldots, a_k$ to denote the edges of this trail and use $a_1$ to denote the edge between $x$ and $x'$. Back to the diagram $K$, note that $a_1\cup a_2$ is an arc of $K$, and $a_1\cup a_2, a_3, \ldots, a_k$ constitute an admissible trail in $K$ connecting $x$ and $y$. As Lemma \ref{4.4} suggests, if we apply arc crossing changes to $a_1\cup a_2, a_3, \ldots, a_k$ one by one, finally we obtain a new knot diagram where only $x$ and $y$ are switched. See Figure \ref{path} for an example.

\begin{figure}[h]
\centering
\includegraphics{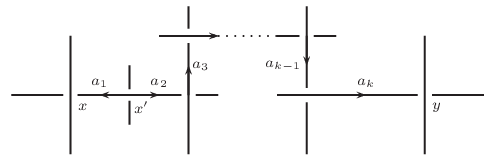}
\caption{A directed path from $x'$ to $y$}
\label{path}
\end{figure}

The proof is finished.
\end{proof}

\begin{remark}
Following the proof of Theorem \ref{the4}, a natural question arises: if every pair of crossings of a knot diagram $K$ is arc crossing change admissible, must $K$ be an alternating knot diagram? The answer to this question is not always affirmative. The knots $8_{19}$ and $8_{21}$ lend support to this question, but $8_{20}$ falsifies it. See Figure \ref{qqq}. Specifically, the diagram of $8_{20}$ is not alternating, yet every pair of its crossings is arc crossing change admissible.
\end{remark}

\begin{figure}[H]
\centering
\includegraphics[width=0.7\linewidth]{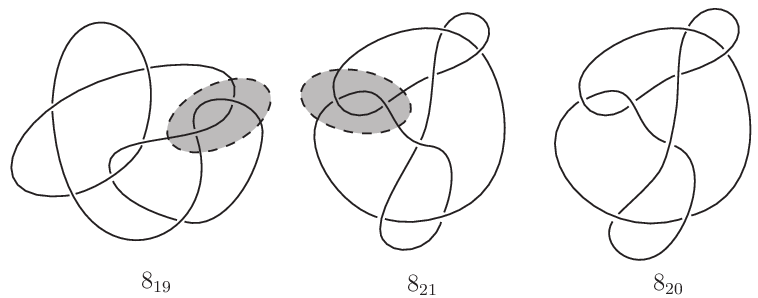}
\caption{Crossings in the shadow area are not arc crossing admissible}
\label{qqq}
\end{figure}

We end this paper with the following question.

\begin{question}
Is the converse of Lemma \ref{4.4} valid? Or more specific, find a necessary and sufficient condition for two crossings in a knot diagram to be arc crossing change admissible.
\end{question}

\section*{Acknowledgement}
Zhiyun Cheng was supported by the NSFC grant 12371065. Qian Liao, Jingze Song were supported by a REU program Grant X202510027207 at Beijing Normal University.


\section*{References}
\begin{biblist}
\bib{Cer2024}{article}{
   author={Cericola, Christopher},
   title={Arc crossing change is an unknotting operation},
   journal={J. Knot Theory Ramifications},
   volume={33},
   date={2024},
   number={1},
   pages={Paper No. 2350091, 14}}

\bib{CCXZ2022}{article}{
    author={Jiawei Cheng},
    author={Zhiyun Cheng},
    author={Jinwen Xu},
    author={Jieyao Zheng},
    title={Region crossing change on surfaces},
    journal={Mathematische Zeitschrift},
    volume={300},
    date={2022},
    number={3},
    pages={2289--2308}}

\bib{CG2012}{article}{
	author={Zhiyun Cheng},
        author={Hongzhu Gao},
	title={On region crossing change and incidence matrix},
	journal={Science China Mathematics},
	volume={55},
	date={2012},
	number={7},
	pages={1487--1495}}

\bib{CS2025}{article}{
	author={Zhiyun Cheng},
        author={Jingze Song},
	title={Region crossing change on nonorientable surfaces},
	journal={arXiv:2506.05885},
	date={2025}}

\bib{Che2013}{article}{
	author={Zhiyun Cheng},
	title={When is region crossing change an unknotting operation?},
	journal={Math. Proc. Cambridge Philos. Soc},
	volume={155},
	date={2013},
	number={2},
	pages={257--269}}

\bib{DR2018}{article}{
	author={Oliver Dasbach},
        author={Heather M. Russell},
	title={Equivalence of edge bicolored graphs on surfaces},
	journal={Electron. J. Combin.},
	volume={25},
	date={2018},
	number={1},
	pages={Paper 1.59, 15 pp.}}
	
\bib{Kin2022}{article}{
       author={Kinuno, Rin},
       title={Structures of homomorphisms induced by arc selection game and arc freeze selection game},
       journal={J. Knot Theory Ramifications},
       volume={31},
       date={2022},
       number={14},
       pages={Paper No. 2250100, 19}}

\bib{Shi2014}{article}{
	author={Ayaka Shimizu},
	title={Region crossing change is an unknotting operation},
	journal={J. Math. Soc. Japan},
	volume={66},
	date={2014},
	number={3},
	pages={693--708}}

\end{biblist}
\end{document}